\documentclass[12pt, reqno]{amsart}
\usepackage{amsmath, amsthm, amscd, amsfonts, amssymb, graphicx, color}
\usepackage[bookmarksnumbered, colorlinks, plainpages]{hyperref}
\RequirePackage{extarrows}
\hypersetup{colorlinks=true,linkcolor=red, anchorcolor=green, citecolor=cyan, urlcolor=red, filecolor=magenta, pdftoolbar=true}

\newtheorem{theorem}{Theorem}[section]
\newtheorem{lemma}[theorem]{Lemma}
\newtheorem{proposition}[theorem]{Proposition}
\newtheorem{corollary}[theorem]{Corollary}
\theoremstyle{definition}

\theoremstyle{remark}

\numberwithin{equation}{section}

\newcommand{\id}{\operatorname{id}}
\newcommand{\lip}{\operatorname{lip}}
\newcommand{\Lip}{\operatorname{Lip}}

\newcommand{\lam}{\Lambda_\alpha(X)}
\newcommand{\lamo}{\Lambda_\alpha^0(X)}
\newcommand{\lbm}{\Lambda_\beta(Y)}
\newcommand{\lbmo}{\Lambda_\beta^0(Y)}
\newcommand{\al}{\alpha}
\newcommand{\var}{\varphi}

\newcommand{\BN}{\mathbb{N}}

\newcommand{\BD}{\mathbb{D}}
\newcommand{\BC}{\mathbb{C}}

\newcommand{\ra}{\rightarrow}

\begin{document}

\setcounter{page}{1}

\title[Weighted Composition Operators ...]{Weighted Composition Operators on Spaces of Analytic Vector-valued   Lipschitz
Functions}

\author[K. Esmaeili ]{K. Esmaeili,$^1$ }

\address{$^{1}$Department of Engineering, Ardakan University, P. O. Box 89518-95491, Ardakan, Yazd, Iran.}
\email{\textcolor[rgb]{0.00,0.00,0.84}{esmaeili@ardakan.ac.ir;
k.esmaili@gmail.com}}


\let\thefootnote\relax



\begin{abstract}
Let $\varphi$  be an analytic self-map  of $\BD$ and  $\psi$ be an analytic operator-valued function  on $\BD$, where $\BD$ is the unit disk.
We provide necessary and sufficient conditions for the boundedness
and compactness of weighted composition operators  $W_{\Psi,\phi}: f\mapsto\Psi(f\circ\phi)$ on $\Lip_A(\overline\BD,X,\al)$ and $\lip_A(\overline\BD,X,\al)$, the spaces of analytic $X$-valued  Lipschitz functions $f$,  where $X$ is a complex Banach space and $\al\in(0,1]$.
\end{abstract} \maketitle

\section{Introduction and preliminaries}


Let $(S,d)$ be  a  metric space, $X$ be a Banach space
and $\alpha \in (0, 1]$. The space of all functions
$f : S \rightarrow X$ for which
$$p_\al(f) = \sup\left\{\frac{\|f(s_1) - f(s_2)\| }{d^\al (s_1,s_2)}: s_1,s_2\in S, s_1\neq s_2\right\}
< \infty,$$
 is denoted by $\Lip_\al(S,X)$. The  subspace of
functions $f$ for which
$$\lim_{d(s_1,s_2)\rightarrow 0} \frac{\|f(s_1)-f(s_2)\|}{d^\al(s_1,s_2)}
=0,$$ is denoted by $\lip_\al (S,X)$. The spaces $\Lip_\al(S,X)$ and
$\lip_\al(S,X)$  equipped with  norm
$\|f\|_\al =\|f\|_S + p_\al(f)$  are Banach spaces, where, 
$$\|f\|_S =\sup_{s\in S}\{\|f(s)\|: s\in S\}.$$
These  spaces are called $X$-valued Lipschitz
spaces.
The spaces $\Lip_\al(S,X)$
and $\lip_\al(S,X)$ were studied by Johnson for the  first time \cite{John}.


For a  Banach space $X$, let $H(\BD, X)$ be the    space of all analytic $X$-valued  functions on the open unit disc $\BD$
 and $A(\overline{\BD},X)$ be   the Banach space of all continuous functions $f:\overline{\BD}\ra X$
which are analytic on $\BD$.
For  $\alpha \in (0, 1]$ we define the spaces  
\begin{align*}
\lam=\Lip_\al(\BD,X)\cap H(\BD,X)
\end{align*}
and
\begin{align*}
\Lip_A(\overline{\BD},X,\al)=\Lip_\al(\overline{\BD},X)\cap A(\overline{\BD},X).
\end{align*}
Clearly,  $\lam$ and $\Lip_A(\overline{\BD},X,\al)$, equipped with  Lipschitz norm $\|\cdot\|_\al$, are Banach spaces and the spaces
\begin{align*}
\lamo=\lip_\al(\BD,X)\cap H(\BD,X)
\end{align*}
and
\begin{align*}
\lip_A(\overline{\BD},X,\al)=\lip_\al(\overline{\BD},X)\cap A(\overline{\BD},X)
\end{align*}
are Banach subspaces of 
 $\lam$ and $\Lip_A(\overline{\BD},X,\al)$, respectively.

We consider the weighted Banach   space
\begin{align*}
H^\infty_\nu(X)=\{f\in H(\BD, X): \|f\|_\nu=\sup_{z\in\BD}\nu(z)\|f(z)\|<\infty\}
\end{align*}
endowed with norm $\|\cdot\|_\nu$, where $\nu:\BD\ra (0,+\infty)$ is a bounded continuous weight function.

 For a positive real number $\al$,  $B_\al(X)$  denotes the $X$-valued Bloch type space of all
 analytic functions $ f : \BD\rightarrow X$  satisfying
$$ \sup_{z\in\BD}(1-|z|^2)^\al\|f'(z)\|<\infty.$$
  The space $B_\al(X)$ endowed with   norm 
  $$\|f\|_{B_\al} =\|f(0)\|+\sup_{z\in\BD}(1-|z|^2)^\al\|f'(z)\|,\ \ (f\in B_\al(X)),$$
  is a Banach space.

In the case   $X=\BC$, we omit $X$ in the notation.

For Banach spaces $X$  and $Y$, by
 $L(X,Y)$ ($K(X,Y)$),  we mean the Banach space of all bounded (compact) linear operators from $X$  to $Y$. 
 Let $S(\overline\BD, X)$ and $S(\overline\BD, Y)$ be  subspaces of $A(\overline\BD, X)$ and $A(\overline\BD, Y)$, respectively. Let $\phi\in A(\overline\BD)$ be a nonconstant self map of $\overline\BD$ and $\Psi:\overline\BD\rightarrow L(X,Y)$  be a continuous operator-valued function analytic on $\BD$. Then the weighted composition operator $W_{\Psi,\phi}$ from $S(\overline\BD, X)$ to $S(\overline\BD, Y)$ is  defined to be the linear operator of the form  
 \begin{align*}
 W_{\Psi,\phi}( f)(z)=\psi(z)(f(\var(z)),\ \ 
(f\in S(\overline\BD, X), \ \ z\in\overline\BD).
\end{align*}
 For simplicity of notation, we write $\Psi_z$ instead of    $\Psi(z)$. 

Note that if  $\Psi_z$ is the identity map on $X$ for every $z\in\overline\BD$, then $W_{\Psi,\phi}$ is the composition operator on $S(\overline\BD,X)$.
In the scalar case, a weighted
composition operator is a composition operator followed by a multiplier.

There is recent interest into properties of the composition operators and weighted composition operators  between Banach  spaces of vector-valued functions.
For instance, the weakly compact composition operators on  Hardy spaces, weighted Bergman spaces, Bloch spaces  and  BMOA, in the vector-valued case, have been characterized in \cite{bonet,laitila4,laitila3,laitila2,saksman}. Weighted composition operators between vector-valued Lipschitz spaces and weighted Banach spaces of vector-valued analytic functions have been studied in 
 \cite{ bjma,laitila}.
Also, composition operators on  analytic Lipschitz spaces in the scalar-valued case   have been  investigated  in   \cite{beh2,san}.
The present study is aimed at finding some necessary and sufficient conditions for    boundedness
and  compactness of  weighted composition operators on the spaces of analytic  $X$-valued Lipschitz functions.

The rest of this paper is designed as follows.  In section 2 we shall  characterize bounded and compact weighted composition operators between analytic vector-valued Lipschitz 
spaces.
Section 3 is devoted to discussing this kinds of operators
 between analytic vector-valued  little Lipschitz  spaces.

\section{Weighted composition operators on $\Lip_A(\overline\BD,X,\al)$ and $\lam$ }
In this section, we provide necessary and sufficient conditions for weighted composition operators between analytic vector-valued Lipschitz spaces to be bounded and compact. In what follows, we will assume that $X$ and $Y$ are Banach spaces. 

We begin with some elementary properties of analytic vector-valued Lipschitz spaces.  The best reference here is \cite{bonet}.
 Let $E$ be  a Banach subspace  of $H(\BD)$ which contains  constant functions and its closed unit ball $U(E)$ is compact for the compact open topology.
Then the space
$$^*E:=\{u\in E^* : u|_{U(E)}\quad is \ \ co-continuous\},$$
endowed with the norm induced by $E^*$, is a Banach space and 
the evaluation map $f\mapsto [u\mapsto u(f)]$ from $E$ into $(^*E)^*$,   is an isometric isomorphism. In particular,   $^*E$ is a predual of $E$.
Furthermore, the vector-valued space
$$E[X]:=\{f\in H(\BD,X): x^*\circ f\in E, \quad x^*\in X^*\},$$ 
 by  the norm $\|f\|_{E[X]}=\sup_{\|x^*\|\leq 1}\|x^*\circ f\|$ is a Banach space. 
 
Bonet, et al. in  \cite[Lemma 10]{bonet} argued that the map $\Delta:\BD\rightarrow  ^*E$, $\Delta(z)=\delta_z$, where $\delta_z$ is the evaluation map on $E$, is analytic and the linear operator  $\chi:L(^*E,X)\rightarrow E[X]$, $\chi(T)=T\circ \Delta$  is bounded.  Defining
  $\psi(g)(u):X^*\rightarrow\BC$ by  $\psi(g)(u)(x^*)=u(x^*\circ g)$ for $g\in E[X]$ and $u\in ^*E$, they showed that 
$\psi(g)\in L(^*E,X^{**})$ and $\psi(g)(\delta_z)\in L(^*E,X)$. Besides, using operators $\chi$ and $\psi$, they  deduced that  
the space $E[X]$ is isomorphic to $L(^*E,X)$. 

In the following proposition we use the above mentioned   result for the spaces $\Lambda_\al[X]$ and $B_{1-\al}[X]$ to show that  the norms $\|\cdot\|_{\al}$  and $\|\cdot\|_{B_{1-\al}(X)}$ are equivalent, whenever $\al\in(0,1)$.

\begin{proposition}\label{pro}
Let $\al\in(0,1)$. Then $\lam=B_{1-\al}(X)$ and the norms $\|\cdot\|_{\al}$  and $\|\cdot\|_{B_{1-\al}}$ are equivalent. 
\end{proposition}
\begin{proof}
Using   Hardy-Littlewood theorem for $\al\in (0,1)$, we can see that
$\Lambda_\al=B_{1-\al}$ and $\|\cdot\|_\al\asymp\|\cdot\|_{B_{1-\al}}$ (i.e.   for some  positive constants $a$ and $ b$,  $a\|\cdot\|_\al\leq \|\cdot\|_{B_{1-\al}}\leq b\|\cdot\|_\al$).
Hence,  $^*\Lambda_\al=^*B_{1-\al}$, where $^*\Lambda_\al$ and $^*B_{1-\al}$ are the   preduals of $\Lambda_\al$ and $B_{1-\al}$ mentioned above, respectively. Thus $\Lambda_\al[X]=B_{1-\al}[X]$ and the linear operators
\begin{align*}
\id:\Lambda_\al[X]\xlongrightarrow{\psi} L(^*\Lambda_\al,X)=  L(^*B_{1-\al},X) \xlongrightarrow{\chi} B_{1-\al}[X]
\end{align*}
and
\begin{align*}
 \id:B_{1-\al}[X]\xlongrightarrow{\psi} L(^*B_{1-\al},X)=L(^*\Lambda_\al,X) \xlongrightarrow {\chi}\Lambda_\al[X]
\end{align*}
 are bounded. This shows that   $\|\cdot\|_{\Lambda_\al[X]}\asymp\|\cdot\|_{B_{1-\al}[X]}$. Since
 $\Lambda_\al(X)=\Lambda_\al[X]$ and  $B_{1-\al}(X)=B_{1-\al}[X]$,  we conclude that
 $$\|\cdot\|_{\al}=\|\cdot\|_{\Lambda_\al[X]}\asymp\|\cdot\|_{B_{1-\al}[X]}=\|\cdot\|_{B_{1-\al}}.$$
 \end{proof}
 
 From Proposition \ref{pro}   we deduce that  for $\al\in(0,1)$
  the norm
 $$\|f\|_{\Lambda_\al(X)}=\|f(0)\|+\sup_{z\in\BD}(1-|z|^2)^{(1-\al)}\|f'(z)\|, \quad (f\in \Lambda_\al(X))$$
 defines an equivalent norm on $\Lambda_\al(X)$. Hereafter we use this norm for $\Lambda_\al(X)$ and $\lamo$, whenever $\al\in(0,1)$.

In \cite{san}, Mahyar and Sanatpour proved that every function $f\in\Lip_\al(\BD)$  has a unique continuous  extension $E(f)$  to  $\overline{\BD}$, such  that $E(f)\in\Lip_\al(\overline{\BD})$. In fact,  for every  $w\in \partial\BD$ they  defined $E(f)(w)=\lim_{n\ra\infty}f(z_n)$, where  $\{z_n\}$ is  any
sequence  in $\BD$ converging to $w$. By the same method, one can  see that 
  every function $f\in\Lip_\al(\BD,X)$ has a unique continuous  extension $E(f)$  to  $\overline{\BD}$ such  that $E(f)\in\Lip_\al(\overline{\BD},X)$.

 Furthermore,  as in the proof of \cite[Proposition \ 2.1]{san}, it can be seen that  the mapping $f\mapsto E(f)$ is a homeomorphism from $(\lam, \|\cdot\|_{\Lambda_\al(X)})$ onto    $( \Lip_A(\overline{\BD},X,\al),\|\cdot\|_\al)$. Thus we can modify  the problem of  boundedness and compactness of the weighted composition operators $W_{\Psi,\phi}: \Lip_A(\overline{\BD},X,\al)\ra\Lip_A(\overline{\BD},Y,\beta),$  to the problem of  boundedness and compactness of 
  the weighted composition operators
$W_{\psi,\var}:\lam\ra\lbm$, where $\psi=\Psi_{|\BD}$ and $\var=\phi_{|\BD}:\BD\ra\BD$.

\begin{proposition}\label{proo}
Let $\Psi\in\Lip_A(\overline{\BD}, L(X,Y),\beta)$ and $\Phi\in A(\overline{\BD})$ be a self map of $\overline{\BD}$. Then 
$W_{\Psi,\phi}: \Lip_A(\overline{\BD},X,\al)\ra\Lip_A(\overline{\BD},Y,\beta)$  is bounded (compact) if and only if 
$W_{\psi,\var}:\lam\ra\lbm$is bounded  (compact), where $\psi=\Psi_{|\BD}$ and $\var=\phi_{|\BD}:\BD\ra\BD$.
\end{proposition}
\begin{proof}
Let $R$ be the restriction map from $\Lip_A(\overline{\BD},X,\al)$ into $\lam$. Clearly, $R$ is a bounded linear operator. 
Let  $z\in\partial\BD$ and  $\{z_n\}$ be any sequence in $\BD$ converging to $z$. An easy computation  shows that for every $f\in\Lip_A(\overline{\BD},X,\al)$,
$$W_{\Psi,\phi}(f)(z)=\lim_{n\ra\infty}\psi_{z_n}f(\var(z_n).$$
On the other hand, 
\begin{align*}
(E\circ  W_{\psi,\var}\circ R)(f)(z)=\lim_{n\ra\infty} W_{\psi,\var}( R(f))(z_n)=
\lim_{n\ra\infty} \psi_{z_n}f(\var(z_n))
\end{align*}
holds for every $f\in\lam$. Thus $E\circ  W_{\psi,\var}\circ R= W_{\Psi,\phi}$ and the diagram
$$\begin{array}[c]{ccc}
 \lam&\stackrel{W_{\psi,\var}}{\longrightarrow}&\lbm\\
\uparrow\scriptstyle{R}&&\downarrow\scriptstyle{E}\\
\Lip_A(\overline{\BD},X,\al)&\stackrel{W_{\Psi,\phi}}{\longrightarrow}&\Lip_A(\overline{\BD},Y,\beta)
\end{array}$$
is   commutative.

Likewise, $W_{\psi,\var}=R\circ  W_{\Psi,\phi}\circ E$ and the diagram
$$\begin{array}[c]{ccc}
\Lip_A(\overline{\BD},X,\al)&\stackrel{W_{\Psi,\phi}}{\longrightarrow}&\Lip_A(\overline{\BD},Y,\beta)\\
\uparrow\scriptstyle{E}&&\downarrow\scriptstyle{R}\\
 \lam&\stackrel{W_{\psi,\var}}{\longrightarrow}&\lbm
\end{array}$$
is  commutative and the proof is complete.
\end{proof}

For every $x\in X$ and  every scalar-valued function $f\in\Lambda_\al$,
the function  $f_x(z)=f(z)x$,  $ z \in\BD$ belongs
to $\lam$ and $\|f_x\|_{\lam}=\|f\|_{\Lambda_\al}\|x\|$.  Moreover, 
\begin{align*}
(W_{\psi,\var}(f_x))'(z)=\var'(z)f'(\var(z))\psi_z(x)+f(\var(z))\psi_z'(x).
\end{align*}
In particular, for  each $x\in X$  the constant function $1_x$
 exists in $\lam$ and the Banach space $X$ can be considered as a subspace
of $\lam$. 

For every vector-valued function  $f\in H(\BD, X)$ and every  $z\in\BD$ we have
\begin{align*}
(W_{\psi,\var}(f))'(z)=\var'(z)\psi_z(f'(\var(z)))+\psi_z'(f(\var(z))).
\end{align*}
Hence,  $DW_{\psi,\var}=W_{\var '\psi,\var}\circ D+W_{\psi ',\var}$, where $D$ is the derivation operator. By \cite[Theorem\ 2.1]{laitila},
for $0<\al<1$ and $0<\beta\leq 1$ we have 
\begin{align}\label{v}\|W_{\psi,\var}:H_{\nu_{\al}}^\infty(X)\rightarrow H_{\nu_{\beta}}^\infty(Y)\|\asymp \sup_{z\in\BD}\frac{(1-|z|^2)^{\beta}}{(1-|\var(z)|^2)^{\al}}\|\psi_z\|,
\end{align}
where $\nu_\al$ is the standard weight $\nu_\al(z)=(1-|z|^2)^\al$.

By the same method, one can see that $W_{\psi,\var}:H^\infty(X)\rightarrow H_{\nu_{\beta}}^\infty (Y)$ is bounded if and only if $\psi\in  H_{\nu_{\beta}}^\infty (L(X,Y))$. Moreover, $\|W_{\psi,\var}:H^\infty(X)\rightarrow H_{\nu_{\beta}}^\infty (Y)\|\asymp \|\psi\|_{  H_{\nu_{\beta}}^\infty}$.

The following  theorem characterizes the  bounded weighted composition operators between analytic vector-valued Lipschitz spaces. 
\begin{theorem} \label{bounded}
For $0<\al< 1$  the operator $W_{\psi,\var}$ maps $\lam$ boundedly into $\lbm$ if and only if $\psi\in \Lambda_\beta(L(X,Y))$ and
\begin{align*} 
q_{\al,\beta}=\sup_{z\in\BD}\frac{(1-|z|^2)^{1-\beta}}{(1-|\var(z)|^2)^{1-\al}}|\var'(z)|\|\psi_z\|<\infty.
\end{align*}
Moreover, 
\begin{align}\label{bndd}
\|W_{\psi,\var}:\lam\rightarrow\lbm\|\asymp\max\{q_{\al,\beta},\|\psi\|_{\Lambda_\beta(L(X,Y))}\}.
\end{align}
\end{theorem}
\begin{proof}
The proof of the first part is a straightforward modification of that of \cite[Theorem\ 2.1]{ohno}. We prove that in the case  $W_{\psi,\var}$ is bounded, the relation (\ref{bndd}) holds. Let $W_{\psi,\var}:\lam\rightarrow\lbm$ be bounded. Using
  $DW_{\psi,\var}=W_{\var '\psi,\var}\circ D+W_{\psi ',\var}$, one can easily show that  the  operators
$W_{\var '\psi,\var}: H_{\nu_{1-\al}}^\infty(X)\rightarrow H_{\nu_{1-\beta}}^\infty(Y)$ and $W_{\psi ',\var}:\lam\rightarrow  H_{\nu_{1-\beta}}^\infty(Y)$ are bounded  and
$$\|W_{\psi ',\var}:\lam\rightarrow  H_{\nu_{1-\beta}}^\infty(Y)\|\leq\frac 2{\al}\|\psi\|_{\Lambda_\beta(L(X,Y))}.$$
Therefore,
\begin{align*}
\|W_{\psi,\var}:\lam\rightarrow\lbm\|&\leq\|W_{\var '\psi,\var}: H_{\nu_{1-\al}}^\infty(X)\rightarrow H_{\nu_{1-\beta}}^\infty(Y)\|\\
&+\|W_{\psi ',\var}:\lam\rightarrow  H_{\nu_{1-\beta}}^\infty(Y)\|.
\end{align*}
From relation (\ref{v}),  for some positive constant $C$ we have  
 $$\|W_{\psi,\var}:\lam\rightarrow\lbm\|\leq C\max\{q_{\al,\beta},\|\psi\|_{\Lambda_\beta(L(X,Y))}\}.$$

For the converse,  since
\begin{align*}
(1-|z|^2)^{1-\beta}\|\psi _z'\|&=\sup_{\|x\|\leq1}(1-|z|^2)^{1-\beta}\|\psi _z '(x)\|\\
&=\sup_{\|x\|\leq1}(1-|z|^2)^{1-\beta}\|(W_{\psi,\var}(1_x))'(z)\|
\end{align*}
holds for every $z\in\BD$ and since $W_{\psi,\var}$ is bounded,
we have  $$\sup_{z\in\BD}(1-|z|^2)^{1-\beta}\|\psi _z'\|\leq\|W_{\psi,\var}\|.$$

For every nonzero  $a\in\BD$ and every $x\in X$ define
$$f_{a,x}(z)=\frac{1}{\overline{a}}\left(\frac{1-|a|^2}{(1-\overline{a}z)^{1-\al}}-(1-\overline{a}z)^{\al}\right)x.$$
 It could be shown that $\{f_{a,x}: a\neq 0, \|x\|\leq1\}$ is a bounded subset of $\lam$. Moreover,  $f_{a,x}(a)=0$ and $f_a'(a)=\frac{x}{(1-|a|^2){1-\al}}$.
Then for some positive constant $C$ we have
$$q_{\al,\beta}=\sup_{\|x\|\leq1}\sup_{z\in\BD}(1-|z|^2)^{1-\beta}\|(W_{\psi,\var}(f_{\var(z),x}))'(z)\|\leq C\|W_{\psi,\var}\|,$$
which completes the proof.
\end{proof}
\begin{theorem}
The operator $W_{\psi,\var}:\Lambda_1(X)\rightarrow\lbm$ is well defined and  bounded if and only if $\psi\in \Lambda_\beta(L(X,Y))$ and
$\var'\psi\in H^\infty_{\nu_{1-\beta}}(L(X,Y))$.
Furthermore,
\begin{align*}
\|W_{\psi,\var}:\Lambda_1(X)\rightarrow\lbm\|\asymp\max\{\|\var'\psi\|_{H^\infty_{\nu_{1-\beta}}},\|\psi\|_{\Lambda_\beta}\}.
\end{align*}
\end{theorem}
\begin{proof}
A simple  computation  gives that $f\in\Lambda_1(X)$ if and only if $f'\in H^\infty(X)$  and $\|f\|_1\asymp \|f\|_{\BD}+ \|f'\|_{\BD}$.

Let $W_{\psi,\var}$ be bounded.  Defining  $f_{a,x}(z)=(z-a)x$, for any $a, z\in\BD$ and $x\in X$, one can see that
$\{f_{a,x}: a\in\BD,\ \ \|x\|\leq1\}$ is a bounded sequence in $\Lambda_1(X)$. Since $W_{\psi,\var}$ is bounded,  for every $z\in\BD$,
\begin{align*}
(1-|z|^2)^{1-\beta}|\var'(z)|\|\psi_z\|&=\sup_{\|x\|\leq 1}(1-|z|^2)^{1-\beta}|\var'(z)|\|\psi_zx\|\\&=
\sup_{\|x\|\leq 1}(1-|z|^2)^{1-\beta}\|(W_{\psi,\var}(f_{\var(z),x}))'(z)\|\\
&\leq \sup_{\|x\|\leq 1}\|W_{\psi,\var}\|\|f_{\var(z),x}\|_{\Lambda_1(X)}<3\|W_{\psi,\var}\|.
\end{align*}
 Thus  $\var'\psi\in H^\infty_{\nu_{1-\beta}}(L(X,Y))$.
The rest of the proof runs as the proof of
 Theorem \ref{bounded}.
\end{proof}


The next lemma shows that the compact open topology on $\Lambda_\al^0(X)$ is stronger than the weak topology.
\begin{lemma}\label{co}
Let $\al\in(0,1)$ and $\{f_n\}$ be a bounded sequence in $\lamo$ converging to zero uniformly on compact subsets of $\BD$. Then
$\{f_n\}$  converges  weakly to zero.
\end{lemma}
\begin{proof}
For every $f\in\lamo$, consider the function 
 $\tilde f(z)=(1-|z|^2)^{1-\al} f'(z)$ and set $\tilde\Lambda=\{\tilde f: f\in\lamo\}$. Clearly, $\tilde\Lambda$
 is a subspace of $C_o(\BD,X)$. Let $T$ be a bounded linear functional on $\lamo$. By Hahn- Banach theorem,  for some measure $\mu\in M(\BD,X^*)$  we have
$
Tf=\int_{\BD}\tilde fd\mu,
$
  for every $f\in\lamo$ (see \cite[Corollary\ 2, p. 387]{din}). Without loss of generality we  assume that
 $\|f_n\|_{\lam}\leq1$. 
Fixing $\epsilon>0$, let $\{r_m\}$ be an increasing sequence in $(0,1)$ converging to $1$ and
$D_m=\{z\in\BD:|z|\leq r_m\}$. Then $\BD=\cup_{m=1}^{\infty}D_m$ and  $|\mu|(\BD\setminus D_m)<\frac{\epsilon}{2}$ for some $m$.
Since $\{f_n\}$ converges to zero uniformly on compact subsets of $\BD$,  we deduce that
 $\sup_{z\in D_m}\|\tilde f_n(z)\|<\frac{\epsilon}{2\|\mu\|}$ for  $n$ sufficiently large. Therefore
\begin{align*}
|T (f_n)|&\leq|\int_{\BD\setminus D_m}\tilde f_nd\mu |+|\int_{ D_m}\tilde f_nd\mu |\\
&\leq\int_{\BD\setminus D_m}\|\tilde f_n(z)\|d|\mu|(z)+\int_{ D_m}\|\tilde f_n(z)\|d|\mu|(z)\\
&\leq|\mu|(\BD\setminus D_m)+|\mu|( D_m)\frac{\epsilon}{2\|\mu\|}<\epsilon,
\end{align*}
which shows that  $\lim_{n\ra\infty}T(f_n)=0$. Thus $\{f_n\}$ converges  weakly to zero as desired.
\end{proof}

In the next theorem we provide a necessary and sufficient condition for 
 weighted composition  operator $W_{\psi,\var}:\lam\rightarrow\lbm$ to be compact, whenever $\al\in(0,1)$.  We use the idea of \cite{laitila} and
define  $T_\psi:X\ra\lbm$, by  $T_\psi(x)(z)=\psi_z(x)$.  In the case  $W_{\psi,\var}$  is bounded,  $T_\psi$ is a
bounded linear operator   and 
$$\|T_\psi:X\ra\lbm\|\leq\|W_{\psi,\var}:\lam\ra\lbm\|.$$
For $n\in\BN$, we define $L_n:\lam\rightarrow\lbm$ by $L_n(f)=\sum_{k=0}^n\frac{f^{(k})(0)}{k!}z^k$, for every $f\in\lam$,
where $\sum_{k=0}^\infty\frac{f^{(k)}(0)}{k!}z^k$ is the Taylor expansion of $f$. One can easily show that for every $f\in\lamo$, 
$\|L_nf-f\|_{\lam}\ra 0$ as $n\ra\infty$.

For $r\in(0,1)$ we define the linear operator $K_r:\lam\ra\lam$, $K_r(f)(z)=f(rz)$ for every $f\in\lam$ and $z\in\BD$. Clearly, $K_r$ is a bounded linear operator from
$\lam$ into $\lamo$ and  for every $f\in\lam$, $\|K_r f-f\|_{\lam\ra 0}$ as $r\ra 1^-$.

\begin{theorem}\label{com}
 Let  $0<\al < 1$  and  $W_{\psi,\var}:\lam\ra\lbm$ be a bounded weighted composition operator.   Then $W_{\psi,\var}$ is compact if and only if  $T_\psi$ is compact  and
 \begin{align}
\limsup_{|\var(z)|\rightarrow1}\frac{(1-|z|^2)^{1-\beta}}{(1-|\var(z)|^2)^{1-\al}}|\var'(z)|\|\psi_z\| = 0.\label{i}
\end{align}
\end{theorem}
\begin{proof}
Let $W_{\psi,\var}:\lam\ra\lbm$ be   compact.  Considering the bounded operator $T: X\rightarrow\lam, Tx=1_x$, we get $T_\psi=W_{\psi,\var}\circ T$ is compact. 
If (\ref{i}) does not hold, one can find a sequence $\{z_n\}$ of $\BD$ such that $|\var(z_n)|>\frac 12$, $|\var(z_n)|\rightarrow 1$ and 
$$\lim_{n\rightarrow\infty}\frac{(1-|z_n|^2)^{1-\beta}}{(1-|\var(z_n)|^2)^{1-\al}}|\var'(z_n)|\|\psi_{z_n}\|>0. $$
For every $n$, define
$$f_n(z)=\frac{1}{\overline{\var(z_n)}}\left(\frac{(1-|\var(z_n)|^2)^2}{(1-\overline{\var(z_n)}z)^{2-\al}}-\frac{1-|\var(z_n)|^2}{(1-\overline{\var(z_n)}z)^{1-\al}}\right)x_n$$
where $\{x_n\}$ is a sequence in $X$, for which $\|x_n\|\leq 1$ and $\frac{n}{n+1}\|\psi_{z_n}\|<\|\psi_{z_n}(x_n)\|$.
Clearly, $\{f_n\}$ is a bounded sequence in $\Lambda_\al^0(X)$ converging to zero on compact  subsets of $\BD$. Moreover,  $f_n(\var(z_n))=0$ and 
$f_n'(\var(z_n))=\frac{1}{(1-|\var(z_n)|^2)^{1-\al}}x_n$. Since $W_{\psi,\var}$ is compact, from lemma \ref{co} we deduce that $W_{\psi,\var}(f_n)\rightarrow 0$, as $n\ra\infty$. But
\begin{align*}
\|W_{\psi,\var}(f_n)\|_{\lbm}&\geq(1-|z_n|^2)^{1-\beta}\|\var '(z_n)f_n'(\var(z_n))\psi_{z_n}(x_n)+f_n(\var(z_n))\psi_{z_n}'(x_n)\|\\
&=\frac{(1-|z_n|^2)^{1-\beta}}{(1-|\var(z_n)|^2)^{1-\alpha}}|\var '(z_n)|\|\psi_{z_n}(x_n)\|\\
&>\frac{(1-|z_n|^2)^{1-\beta}}{(1-|\var(z_n)|^2)^{1-\alpha}}|\var '(z_n)|\|\psi_{z_n}\|\frac{n}{n+1},
\end{align*}
implies that  $\lim_{n\rightarrow\infty}\|W_{\psi,\var}(f_n)\|_{\lbm}>0$, which is impossible.

Conversely, let  $T_\psi$ be  compact and (\ref{i})  holds. Consider the operator $q_k:\lam\rightarrow X$, $f\mapsto\frac{f^{(k)}(0)}{k!}$.
By Cauchy's integral theorem,  for every $f\in\lam$ we have
$$\|q_k(f)\|\leq\frac{1}{2\pi}\oint_{|z|=\frac 12}\frac{\|f'(z)\|}{|z|^k}d|z|\leq\frac{2^k}{C}\|f\|_{\lam},$$
where $C$ is a positive constant. Therefore $q_k$ is a bounded linear operator.
We show that $M_{\var^k}:\lbm\rightarrow\lbm$, $f\mapsto \var^k f$ is  bounded.
Since $W_{\psi,\var}:\lam\ra\lbm$ is bounded, from Theorem \ref{bounded} we conclude  that
\begin{align*}
\sup_{z\in\BD}\frac{(1-|z|^2)^{1-\beta}}{(1-|\var(z)|^2)^{1-\al}}|\var '(z)|<\infty.
\end{align*}
Thus   for every  $f\in\lbm$ we have
\begin{align*}
\| M_{\var ^k}(f)\|_{\lbm}&=\sup_{z\in\BD}(1-|z|^2)^{1-\beta}\|(\var^k f)'(z)\|\\
&\leq\sup_{z\in\BD}(1-|z|^2)^{1-\beta}\| f'(z)\|+k\sup_{z\in\BD}(1-|z|^2)^{1-\beta}|\var'(z)|\| f(z)\|\\
&\leq\| f\|_{\lbm }+ k\sup_{z\in\BD}\frac{(1-|z|^2)^{1-\beta}}{(1-|\var(z)|^2)^{1-\al}}(1-|\var(z)|^2)^{1-\al}|\var'(z)|\|f(z)\|\\
&\leq(2+kC)\|f\|_{\lbm},
\end{align*}
where $C$ is a positive constant.  This shows that $M_{\var^k}$ is a bounded  operator on $\lbm$ and since  $T_\psi$ is compact, we deduce that
$W_{\psi,\var}\circ L_n=\sum_{k=0}^n M_{\var^k}\circ T_\psi\circ q_k$
 is a compact  operator. 
Since  $K_r f\in\lamo$, we see that 
$\|K_r f-L_n(K_rf)\|_{\lam}\rightarrow 0$ as $n\rightarrow\infty$ . Hence   $\|W_{\psi,\var}\circ K_r-W_{\psi,\var}\circ L_n\circ K_r\|\rightarrow 0$ as $n\rightarrow\infty$, which shows that $W_{\psi,\var}\circ K_r$ is a compact operator. For completing the proof we show that 
$\limsup_{r\rightarrow {1^-}}\|W_{\psi,\var}-W_{\psi,\var}\circ K_r\|=0$. For this, fix $0<\delta<1$.  For every $f\in\lam$ with
$\|f\|_{\lam}\leq1$ we have
\begin{align}
&\sup_{z\in\BD}(1-|z|^2)^{1-\beta}\|(W_{\psi,\var}(f-K_rf))'(z)\|\nonumber\\ &\leq 
\sup_{z\in\BD}(1-|z|^2)^{1-\beta}|\var'(z)|\|\psi_z\|\|f'(\var(z))-rf'(r\var(z))\|\nonumber\\
&+\sup_{z\in\BD}(1-|z|^2)^{1-\beta}\|\psi_z'\|\|f(\var(z))-f(r\var(z))\|\nonumber\\
&\leq \sup_{|\var(z)|\leq\delta}(1-|z|^2)^{1-\beta}|\var'(z)|\|\psi_z\|\|f'(\var(z))-f'(r\var(z))\|\label{x}\\
&+\sup_{|\var(z)|\leq\delta}(1-|z|^2)^{1-\beta}|\var'(z)|\|\psi_z\|\|f'(r\var(z))\|(1-r)\label{xx}\\
&+\sup_{|\var(z)|>\delta}(1-|z|^2)^{1-\beta}|\var'(z)|\|\psi_z\|\|(f-f_r)'(\var(z))\|\label{xxx}\\
&+\sup_{z\in\BD}(1-|z|^2)^{1-\beta}\|\psi_z'\|\|f(\var(z))-f(r\var(z))\|.\label{xxxx}
\end{align}
By Cauchy's integral theorem,  (\ref{x}) is not bigger than $\frac{(1-r)}{(1-\delta)^{3-\al}}$ and converges to zero whenever  $r\rightarrow 1^{-}$. Also,
\begin{align*}
(\ref{xx})\leq\sup_{|\var(z)|\leq\delta}\frac{(1-|z|^2)^{1-\beta}}{(1-|\var(z)|^2)^{1-\alpha}}|\var'(z)|\|\psi_z\|(1-r)\rightarrow 0,
\end{align*}
as $r\rightarrow 1^{-}$ and
\begin{align*}
(\ref{xxx})\leq 2\sup_{|\var(z)|>\delta}\frac{(1-|z|^2)^{1-\beta}}{(1-|\var(z)|^2)^{1-\alpha}}|\var'(z)|\|\psi_z\|.
\end{align*}
Moreover,
\begin{align*}
(\ref{xxx})\leq\sup_{z\in\BD}\|\psi_z\|\|f\|_{\al}|\var(z)|(1-r)^\alpha\ra 0,
\end{align*}
as $r\rightarrow 1^{-}$.
Thus
\begin{align*}
\lim_{r\rightarrow {1^-}}\|W_{\psi,\var}-W_{\psi,\var}\circ K_r\|=\lim_{r\rightarrow {1^-}}\sup_{\|f\|\leq 1}\|W_{\psi,\var}(f)-W_{\psi,\var}\circ K_r(f)\|\\
\leq 2\sup_{|\var(z)|>\delta}\frac{(1-|z|^2)^{1-\beta}}{(1-|\var(z)|^2)^{1-\alpha}}|\var'(z)|\|\psi_z\|.
\end{align*}
Letting $\delta\rightarrow 1$, we have
\begin{align*}
\lim_{r\rightarrow 1^{-}}\|W_{\psi,\var}-W_{\psi,\var}\circ K_r\|\leq2\lim_{\delta\rightarrow1}\sup_{|\var(z)|>\delta}\frac{(1-|z|^2)^{1-\beta}}{(1-|\var(z)|^2)^{1-\alpha}}|\var'(z)|\|\psi_z\|=0,
\end{align*}
which ensures that $W_{\psi,\var}$ is compact.
\end{proof}


\section{Weighted composition operators on  $\lip_A(\overline\BD,X,\al) $ and $\lamo$ }
In this section we characterize   bounded and compact weighted composition operators  on the spaces of analytic  vector-valued little  Lipschitz functions. 

 Ohno, et al. in \cite [Theorem\  4.1]{ ohno} showed that for $\al\in(0,1)$,   the operator $W_{\psi,\var}:\Lambda_\alpha ^0\rightarrow\Lambda_\beta ^0$
is bounded if and only if $\psi\in\Lambda_\beta^0$, $W_{\psi,\var}:\Lambda_\alpha \rightarrow\Lambda_\beta $ is bounded and 
$\lim_{|z|\rightarrow 1^{-}}(1-|z|^2)^{1-\beta}\var'(z)\psi_z=0$. For the  vector-valued case, we have weaker results. That is, if  $W_{\psi,\var}:\lamo\rightarrow\lbmo$ is well defined and bounded, then
 $\psi\in\Lambda_\beta(\BD, L(X,Y))$   and the point wise limit of $(1-|z|^2)^{1-\beta}\var'(z)\psi_z$ is zero, whenever $|z|\rightarrow 1^{-}$. 

By the next theorem, we provide some  sufficient conditions for the weighted composition operator $W_{\psi,\var}:\lamo\rightarrow\lbmo$ to be well defined and bounded.
\begin{theorem}
Let $W_{\psi,\var}:\lam\rightarrow\lbm$  be bounded, $\psi\in\Lambda_\beta^0(\BD, L(X,Y))$ and 
$\lim_{|z|\rightarrow 1^{-}}(1-|z|^2)^{1-\beta}\var'(z)\psi_z=0$. Then $W_{\psi,\var}:\lamo\rightarrow\lbmo$ is well defined and bounded.
\end{theorem}
\begin{proof}
We just show that $W_{\psi,\var}:\lamo\rightarrow\lbmo$ is well defined.  The boundedness can be shown by means of the  closed graph theorem. 
 Let $f\in\lamo$. Given $\epsilon>0$, for some $r\in(0,1)$ and for every $z\in\BD$ with $r<|z|<1$, we have
 \begin{align*}
&(1-|z|^2)^{1-\beta}\|\psi_z'\|<\frac{\al\epsilon}{4\|f\|_{\lam}},\\
&(1-|z|^2)^{1-\alpha}\|f'(z)\|<\frac{\epsilon}{4M},\\
&(1-|z|^2)^{1-\beta}|\var'(z)|\|\psi_z\|<\frac{\epsilon}{4L}, 
\end{align*}
where 
\begin{align*}M=\sup_{z\in\BD}\frac{(1-|z|^2)^{1-\beta}}{(1-|\var(z)|^2)^{1-\alpha}}|\var'(z)|\|\psi_z\|<\infty,
\end{align*}
 and 
 \begin{align*}
 L=\max\{\sup_{|z|\leq r}\|f(z)\|,\sup_{|z|\leq r}\|f'(z)\|\}<\infty.
 \end{align*}
  Fix $z\in\BD$ with $r<|z|<1$. If $r<|\var(z)|<1$, then
\begin{align*}
&(1-|z|^2)^{1-\beta}|\var'(z)|\|\psi_z(f'(\var(z)))\|\\
&\leq \frac{(1-|z|^2)^{1-\beta}}{(1-|\var(z)|^2)^{1-\alpha}}
(1-|\var(z)|^2)^{1-\alpha}|\var'(z)|\|\psi_z\|\| f'(\var(z))\|<\frac{\epsilon}{4},
\end{align*}
and for $|\var(z)|\leq r$ we have
\begin{align*}
(1-|z|^2)^{1-\beta}|\var'(z)|\|\psi_z(f'(\var(z)))\|\leq L(1-|z|^2)^{1-\beta}
|\var'(z)|\|\psi_z\|<\frac{\epsilon}{4}.
\end{align*} 
This shows that
\begin{align*}
\sup_{r<|z|<1}(1-|z|^2)^{1-\beta}|\var'(z)|\|\psi_z(f'(\var(z)))\|<\frac{\epsilon}{2}.
\end{align*}
It is easy to check that  $\|f((\var(z))\|\leq\frac{1}{\al}\|f\|_{\lam}$.
Thus
\begin{align*}
\sup_{r<|z|<1}(1-|z|^2)^{1-\beta}\|(W_{\psi,\var}(f))'(z)\|& \leq
 \sup_{r<|z|<1}(1-|z|^2)^{1-\beta}|\var'(z)|\|\psi_z(f'(\var(z)))\|\\
 &+\sup_{r<|z|<1}(1-|z|^2)^{1-\beta}\|\psi_z'(f(\var(z)))\|
<\epsilon,
\end{align*}
which ensures  that $W_{\psi,\var}(f)\in\lbmo$. 
\end{proof}
For characterizing the compact weighted composition operators between analytic little Lipschitz spaces we need the next lemma.
\begin{lemma}\label{ascoli}
A subset  $K$ of $\lip_A(\overline{\BD},X,\al)$ is relatively compact if and only if
\begin{itemize}
\item[(i)] $K$ is bounded,
\item[(ii)] $K(z)=\{f(z):\ \ f\in K\}$ is relatively compact for every $z\in\overline\BD$,
\item[(iii)] $\lim\limits_{\substack{|z|\rightarrow 1^-\\ z\in\BD}}\sup_{f\in K}(1-|z|^2)^{1-\alpha}\|f'(z)\|=0.$
\end{itemize}
\end{lemma}
\begin{proof}
We begin by proving that (iii) is necessary. Suppose that $K$ is relatively compact and let $C$ be a positive constant such that $\|\cdot\|_{\lam}\leq C\|\cdot\|_\al$. Given $\epsilon>0$, there are functions $f_1, f_2, \cdots, f_n\in K$ such that for every $f\in K$ and for some $1\leq j\leq n$, we have
$\|f-f_j\|_{\al}<\frac{\epsilon}{2C}$ which ensures  that 
$$(1-|z|^2)^{1-\alpha}\|f'(z)\|\leq (1-|z|^2)^{1-\alpha}\|f_j'(z)\|+\frac{\epsilon}{2}, \ \  (z\in\BD).$$
 For each $1\leq j\leq n$, there exists $r_j\in(0,1)$
 such that 
 $$\sup_{r_j<|z|<1}(1-|z|^2)^{1-\alpha}\|f_j'(z)\|<\frac{\epsilon}{2}.$$ 
 Setting $r=\max\{r_1, \cdots, r_n\}$ yields the assertion
 $$\sup_{r<|z|<1}(1-|z|^2)^{1-\alpha}\| f'(z)\|\leq\epsilon, \ \ \  (f\in K)$$
 and (iii) holds. 
 
 For the converse, let $\{f_n\}$ be a bounded sequence in $\lip_A(\overline{\BD},X,\al)$. Hence, $\{f_n\}$ is an equicontinuous sequence in $C(\overline\BD,X)$. From (ii) and by generalized 
 Arzela- Ascoli theorem, \cite [Theorem A]{chan},  $\{f_n\}$ is relatively compact in $C(\overline\BD,X)$. Thus  there exists a  subsequence $\{f_{n_k}\}$ of $\{f_n\}$ which is  Cauchy in $C(\overline\BD,X)$.  Let $g_n=f_{n}|_{\BD}$.
 We show that $\{g_{n_k}\}$ is a Cauchy sequence  in $\lamo$. 
 Fix $\epsilon>0$.  By (iii),  there exists   $0<r<1$ such that for every $j$,
 $$(1-|z|^2)^{1-\alpha}\| g_{n_j}'(z)\|<\frac \epsilon 4, \ \ \ \ (r<|z|<1).$$
  It is  easy to check  that
 $\{g_{n_k}'\}$ is Cauchy with respect to compact open topology on $C(\BD,X)$. Thus  for $k$ and $ l$ sufficiently  large, we have
   $$\sup_{|z|\leq r}\|g_{n_k}'(z)-g_{n_l}'(z)\|<\frac{\epsilon}{2}.$$
Hence,
 \begin{align*}
 \sup_{z\in\BD}(1-|z|^2)^{1-\alpha} \|g_{n_k}'(z)-g_{n_l}'(z)\|&\leq \sup_{|z|\leq r}(1-|z|^2)^{1-\alpha}\|g_{n_k}'(z)-g_{n_l}'(z)\|\\
 &+\sup_{r<|z|<1}(1-|z|^2)^{1-\alpha}\|g_{n_k}'(z)-g_{n_l}'(z)\|\\
 &<\epsilon,
 \end{align*}
 which implies that $\{g_{n_k}\}$ is a Cauchy sequence in $\lamo$. Thus $\{f_n\}$ is Cauchy  in $\lip_A(\overline{\BD},X,\al)$, as desired.
\end{proof}
Regarding the  arguments following the proof of Proposition \ref{pro}, for every $\psi\in\Lambda_\beta(K(X,Y))$ there exists an operator-valued function $\Psi\in\Lip_A(\overline{\BD},L(X,Y),\beta)$ such that $\Psi_{|\BD}=\psi$. More precisely, for every $z\in\partial\BD$, $\Psi_z=\lim_{n\ra\infty}\psi_{z_n}$, where $\{z_n\}$ is any sequence in $\BD$ converging to $z$. Since for every $n$,  $\psi_{z_n}:X\ra Y$  is compact, we deduce that  $\Psi_{z}$ is a compact linear operator from $X$ in to $Y$ and $\Psi\in\Lip_A(\overline{\BD},K(X,Y),\beta)$.

Now we can  characterize the compact weighted composition operators from  $\lamo$ into $\lbmo$.
\begin{theorem}\label{compactness}
Let $W_{\psi,\var}:\lamo\rightarrow\lbmo$ be a bounded weighted composition operator.
\begin{itemize}
\item[(i)] If $W_{\psi,\var}$ is compact, then $\psi\in\Lambda_\beta(K(X,Y))$ and 
\begin{align}\label{y}
\lim_{|z|\rightarrow 1^-}\frac{(1-|z|^2)^{1-\beta}}{(1-|\var(z)|^2)^{1-\alpha}}|\var'(z)|\|\psi_z\|=0.
\end{align}
\item[(ii)] If $\psi\in\Lambda_\beta^0(K(X,Y))$ and (\ref{y}) holds, then $W_{\psi,\var}$ is compact.
\end{itemize}
\end{theorem}
\begin{proof}
(i) Let $U_\al$ be the closed unit ball of $\lamo$. Since  $W_{\psi,\var}$ is compact,  $W_{\psi,\var}(U_\al)$ is relatively compact in $\lbmo$ and hence,
$E(W_{\psi,\var}(U_\al))$ is relatively compact in $\lip_A(\overline{\BD},Y,\beta)$. By Lemma \ref{ascoli}, we have
\begin{align*}
\lim_{|z|\rightarrow 1^-}\sup_{f\in U_\al}(1-|z|^2)^{1-\beta}\|(W_{\psi,\var}(f))'(z)\|=0.
\end{align*}
Given $\epsilon>0$, there exists $r\in(0,1)$ such that 
\begin{align}\label{yy}
\sup_{r<|z|<1}\sup_{f\in U_\al}(1-|z|^2)^{1-\beta}\|(W_{\psi,\var}(f))'(z)\|<\frac{\epsilon}{3}.
\end{align}
For every $x\in X$ and every nonzero $a\in\BD$, define 
$$f_{a,x}(z)=\frac 1{\overline a}\left(\frac{1-|a|^2}{(1-\overline a z)^{1-\alpha}}-(1-\overline a z)^{\alpha}\right)x.$$
One can  see that $\{f_{a,x}:0\neq a, \|x\|\leq 1\}$ is a bounded subset of $\lamo$. Moreover,
 $f_{a,x}(a)=0$, $f_{a,x}'(a)=\frac{x}{(1-|a|^2)^{1-\alpha}}$ and $\sup_{\substack{\|x\|\leq 1\\ z\in\BD}}\|f_{a,x}\|_{\lam}\leq 3$.   Since $(W_{\psi,\var}(f_{\var(z),x}))'(z)=\frac{\var'(z)}{(1-|\var(z)|^2)^{1-\alpha}}\psi_z(x)$, by relation (\ref{yy}) we have
 \begin{align*}
\sup_{r<|z|<1}\frac{(1-|z|^2)^{1-\beta}}{(1-|\var(z)|^2)^{1-\alpha}}|\var'(z)|\|\psi_z\|&
=\sup_{r<|z|<1}\sup_{\|x\|\leq1}\frac{(1-|z|^2)^{1-\beta}}{(1-|\var(z)|^2)^{1-\alpha}}|\var'(z)|\|\psi_zx\|\\
&=\sup_{r<|z|<1}\sup_{\|x\|\leq1}(1-|z|^2)^{1-\beta}\|(W_{\psi,\var}(f_{\var(z),x}))'(z)\|\\&<\epsilon.
\end{align*}
(ii) Let $\psi\in\Lambda_\beta^0(K(X,Y))$ and (\ref{y}) holds. Fixing $\epsilon>0$, for some 
$r\in(0,1)$ we have
\begin{align*}
\sup_{r<|z|<1}\frac{(1-|z|^2)^{1-\beta}}{(1-|\var(z)|^2)^{1-\alpha}}|\var'(z)|\|\psi_z\|&<\frac{\epsilon}{8},\end{align*}
and
\begin{align*}
\sup_{r<|z|<1}(1-|z|^2)^{1-\beta}\|\psi'_z\|<\frac{\al\epsilon}{8}.
\end{align*}
Let $\{f_n\}$ be a bounded sequence in $\lamo$ such that $\|f_n\|_{\lam}\leq 1$. 
Then  $\{E(f_n)\}$ is a bounded sequence in $\lip_A(\overline{\BD},X,\al)$ and $\{W_{\Psi,\phi}(E(f_n))\}$ is a bounded sequence in $\lip_A(\overline{\BD},Y,\beta)$, which  implies that $\{W_{\Psi,\phi}(E(f_n))\}$ is equicontinuous. For every $z\in\overline\BD$, 
$\{E(f_n)(\phi(z))\}$ is a bounded sequence in $X$ and $\Psi_z:X\ra Y$ is a compact operator. Hence, $\{W_{\Psi,\phi}(E(f_n))(z)\}$
is relatively compact in $Y$.
From  Arzela-Ascoli theorem we deduce that  $\{W_{\Psi,\phi}(E(f_n))\}$ is relatively compact in $C(\overline{\BD},Y)$. Therefore, there exists a subsequence  $\{f_{n_k}\}$ of $\{f_n\}$ such that $\{W_{\psi,\var}(f_{n_k})\}$ is uniformly convergent 
on compact subsets of $\BD$.  Using Cauchy's integral theorem, one can show that $\{(W_{\psi,\var}(f_{n_k}))'\}$ is convergent with respect  to the compact open topology. Thus
for $k$ and $ l$ sufficiently large,
\begin{align*}
\sup_{|z|\leq r}\|(W_{\psi,\var}(f_{n_k}))'(z)-(W_{\psi,\var}(f_{n_l}))'(z)\|<\frac{\epsilon}{4}.
\end{align*}
Therefore,
\begin{align*}
&\sup_{z\in\BD}
(1-|z|^2)^{1-\beta}\|(W_{\psi,\var}(f_{n_k}))'(z)-(W_{\psi,\var}(f_{n_l}))'(z)\|\\
&\leq \sup_{|z|\leq r}(1-|z|^2)^{1-\beta}\|(W_{\psi,\var}(f_{n_k}))'(z)-(W_{\psi,\var}(f_{n_l}))'(z)\|\\
&+\sup_{r<|z|<1}(1-|z|^2)^{1-\beta}|\var'(z)|\|\psi_z(f_{n_k}'(\var(z))-f_{n_l}'(\var(z)))\|\\
&+\sup_{r<|z|<1}(1-|z|^2)^{1-\beta}\|\psi_z'(f_{n_k}(\var(z))-f_{n_l}(\var(z)))\|\\
&<\frac{\epsilon}{4}+2\sup_{r<|z|<1}\frac{(1-|z|^2)^{1-\beta}}{(1-|\var(z)|^2)^{1-\alpha}}|\var'(z)|\|\psi_z\|\\
&+
\frac{2}{\al}\sup_{r<|z|<1}(1-|z|^2)^{1-\beta}\|\psi_z'\|<\epsilon.
\end{align*}
We  conclude that $\{W_{\psi,\var}(f_{n_k})\}$ is a Cauchy sequence  in $\lbmo$ and hence $W_{\psi,\var}$ is a compact operator.
\end{proof}
The next corollary is  an immediate consequence of 
  Theorem \ref{compactness}.
\begin{corollary} Let    $W_{\Psi,\phi}:\lip_A(\overline{\BD},X,\al)\ra\lip_A(\overline{\BD},Y,\beta)$  be a bounded weighted composition operator.
\begin{itemize}
\item[(i)] If $W_{\Psi,\phi}$ is compact, then $\Psi\in\Lip_A(\overline{\BD},K(X,Y),\beta)$ and 
\begin{align}\label{zzz}
\lim_{|z|\rightarrow 1^-}\frac{(1-|z|^2)^{1-\beta}}{(1-|\var(z)|^2)^{1-\alpha}}|\var'(z)|\|\psi_z\|=0.
\end{align}
\item[(ii)] If $\Psi\in\lip_A(\overline{\BD},K(X,Y),\beta)$ and (\ref{zzz}) holds, then $W_{\Psi,\phi}$ is compact.
\end{itemize}
\end{corollary}
\bibliographystyle{amsplain}

\end{document}